\documentclass[a4paper,10pt,twoside]{amsart}
\usepackage{amsthm}
\usepackage{amsmath}
\usepackage{amssymb}
\usepackage{amsfonts}
\usepackage{amscd}
\usepackage[english]{babel}
\usepackage{stmaryrd}
\usepackage{enumerate}

\newtheorem{thm}{Theorem}
\newtheorem{cor}{Corollary}
\newtheorem{prop}{Proposition}
\newtheorem{rem}{Remark}

\newtheorem{lem}{Lemma}

\newtheorem*{thm2}{Theorem}

\begin{document}

    \title{About the matrix function $X\rightarrow AX+XA$}
    \author{Gerald BOURGEOIS}
    
    \date{9-30-2012}
    \address{G\'erald Bourgeois, GAATI, Universit\'e de la polyn\'esie fran\c caise, BP 6570, 98702 FAA'A, Tahiti, Polyn\'esie Fran\c caise.}
    \email{bourgeois.gerald@gmail.com}
        

\begin{abstract}
Let $K$ be an infinite field such that $\mathrm{char}(K)\not= 2$. We show that, for every $A\in\mathcal{M}_n(K)$ such that $\mathrm{rank}(A)\geq n/2$, there exists $B\in\mathcal{M}_n(K)$ such that $B$ is similar to $A$ and $A+B$ is invertible. Let $K$ be a subfield of $\mathbb{R}$. We show that, if $n$ is even, then for every $X\in\mathcal{M}_n(K)$, $\det(AX+XA)\geq 0$ if and only if either $\mathrm{rank}(A)<n/2$ or there exists $\alpha\in K,\alpha\leq 0$, such that $A^2=\alpha I_n$.

\end{abstract}

\maketitle

    \section{Introduction}
    Let $K$ be a field and $n\in\mathbb{N}_{\geq 2}$. If $M\in\mathcal{M}_n(K)$, the $n\times n$ matrices with entries in $K$, then $\mathrm{adj}(M)$ denotes its classical adjoint, $\mathrm{tr}(M)$ denotes its trace and $\mathrm{sp}(M)$ denotes its spectrum in $\overline{K}$, an algebraic closure of $K$. Consider the function $\phi:X\in\mathcal{M}_n(K)\rightarrow \det(AX+XB)\in K$. \\
    Problem 1. Characterize the matrices $A\in\mathcal{M}_n(K)$ such that 
     $$\phi=0\text{ (resp. }\phi\geq 0,\text{ resp. }\phi\leq 0\text{ when }K=\mathbb{C}\text{ or }\mathbb{R}).$$
    Problem 2. Characterize the matrices $A\in\mathcal{M}_n(K)$ such that there exist $B\in\mathcal{M}_n(K)$ such that $B$ is similar to $A$ and $A+B$ is invertible. It is linked to Problem 1, because $\phi(X)\not=0$ and $X$ invertible imply $A+XAX^{-1}$ invertible.\\ 
       To solve Problem 1, we can consider $\mathrm{adj}(AX+XA)$. We obtain the following results\\
    $i)$ If $\mathrm{rank}(A)<n/2$ then $\phi=0$ and for every $X$, $\mathrm{adj}(AX+XA)\times A=A\times\mathrm{adj}(AX+XA)=0_n$.\\
    $ii)$ In the particular case when $K=\mathbb{C}$ or $\mathbb{R}$ : $\phi=0\;\Leftrightarrow$ for every $X\in\mathcal{M}_n(K)$, $\mathrm{adj}(AX+XA)\times A+A\times\mathrm{adj}(AX+XA)=0_n$.\\
   Problem 2 is connected to the following Roth's result, valid over any field $K$ (cf. \cite{4})
  \begin{thm2}  Let $A,A',C\in\mathcal{M}_n(K)$. The matrix equation $AX-XA'=C$ has a solution if and only if the matrices $\begin{pmatrix}A&C\\0&A'\end{pmatrix}$ and $\begin{pmatrix}A&0\\0&A'\end{pmatrix}$ are similar.
  \end{thm2}
  Then we seek the matrices $A$ such that there exists an invertible matrix $C$ such that  $\begin{pmatrix}A&C\\0&-A\end{pmatrix}$ and $\begin{pmatrix}A&0\\0&-A\end{pmatrix}$ are similar.\\
  We can obtain better characterizations when the underlying field $K$ is infinite and has a characteristic that is not $2$. Under this hypothesis, our first main result is\\
  $\bullet$ If $\mathrm{rank}(A)\geq n/2$, then there exists $B$ such that $B$ is similar to $A$ and $A+B$ is invertible. As a corollary, we show that $\phi=0\;\Leftrightarrow\;\mathrm{rank}(A)< n/2$.\\
  $\bullet$ When $K$ is a subfield of $\mathbb{R}$ and $n$ is even, we show our second main result 
   $$\phi\geq 0\; \Leftrightarrow\text{ either }\mathrm{rank}(A)<n/2\text{ or there exists }\alpha\in K,\alpha\leq 0,\text{ such that }A^2=\alpha I_n.$$
   As a corollary (valid for any $n$), we show that $\phi\leq 0\;\Leftrightarrow\;\mathrm{rank}(A)<n/2$.
  
  \section{About $\det(AX+XA)$}
  First, we show the two results concerning the matrix $\mathrm{adj}(AX+XA)$.
   \begin{lem}  \label{odd}
 Let $A\in\mathcal{M}_n(K)$. If $\mathrm{rank}(A)<n/2$, then, for every $X\in\mathcal{M}_n(K)$, $\det(AX+XA)=0$ and $\mathrm{adj}(AX+XA)\times A=A\times\mathrm{adj}(AX+XA)=0_n$. 
 \end{lem} 
 \begin{proof}
 $\bullet$ Since $\mathrm{rank}(A)\leq\dfrac{n-1}{2}$, one has $\mathrm{rank}(AX+XA)\leq n-1$. Thus $\det(AX+XA)=0$.\\
 $\bullet$ If $\mathrm{rank}(AX+XA)<n-1$ then $\mathrm{adj}(AX+XA)=0$. Thus assume that $\mathrm{rank}(AX+XA)=n-1$. Then $n$ is odd, $r=\mathrm{rank}(A)=\dfrac{n-1}{2}$ and $\mathrm{rank}(\mathrm{adj}(AX+XA))=1$. 
  Assume that, for every $u\in\ker(A)\setminus\{0_{n,1}\}$, $Xu\not= 0_{n,1}$. Then $X(\ker(A))$ is a vector subspace of $K^n$ of dimension $\dfrac{n+1}{2}$ and therefore intersects $\ker(A)\setminus\{0_{n,1}\}$. Finally, there exists $u\in\ker(A)\setminus\{0_{n,1}\}$ such that $Xu\in\ker(A)$ and thus $(AX+XA)u=0$. Again the vector subspace $\ker(A^T)$ has dimension $\dfrac{n+1}{2}$. In the same way, we find $v\in\ker(A^T)\setminus\{0_{n,1}\}$ such that $X^Tv\in\ker(A^T)$. Then $(AX+XA)^Tv=0$. Therefore (cf. \cite{1}, p. 41), there exists $\lambda\in K^*$ such that $\mathrm{adj}(AX+XA)=\lambda uv^T$. Clearly $A\times \mathrm{adj}(AX+XA)=\mathrm{adj}(AX+XA)\times A=0_n$. 
 \end{proof}
 \begin{lem} Assume that $K=\mathbb{C}$ or $\mathbb{R}$.
 Let $A\in\mathcal{M}_n(K)$  Then the following conditions are equivalent.\\
 $i)$ For every $X\in\mathcal{M}_n(K)$, $\det(AX+XA)=0$.\\
 $ii)$ For every $X\in\mathcal{M}_n(K)$, $\mathrm{adj}(AX+XA)\times A+A\times\mathrm{adj}(AX+XA)=0_n$.
  \end{lem}
 \begin{proof} Let $D\phi$ be the derivative of $\phi$. Then
  $$ \phi=0\;\Leftrightarrow \text{ for every }X\in\mathcal{M}_n(K),\;\; D\phi(X)=0$$ 
 $$\Leftrightarrow \text{ for every }X,H\in\mathcal{M}_n(K),\;\;\mathrm{tr}(\mathrm{adj}(AX+XA)(AH+HA))=0$$
  $$\Leftrightarrow \text{ for every }X,H\in\mathcal{M}_n(K),\;\;\mathrm{tr}((A\times \mathrm{adj}(AX+XA)+\mathrm{adj}(AX+XA)\times A)H)=0$$
    $$\Leftrightarrow \;ii).$$
     \end{proof}
      Let $\mathrm{char}(K)$ be the characteristic of $K$. We show our first main result.
   \begin{thm}  \label{florian}
  Let $K$ be an infinite field such that $\mathrm{char}(K)\not=2$. Let $A\in\mathcal{M}_n(K)$ be such that $\mathrm{rank}(A)\geq n/2$. Then there exists $B\in\mathcal{M}_n(K)$ such that $B$ is similar to $A$ and $A+B$ is invertible.  
  \end{thm}
  \begin{proof} 
  $\bullet$ Case 1. $K$ is algebraically closed. \\
    Step 1. $\mathrm{rank}(A)\geq n/2$ and $\mathrm{ker}(A)=\mathrm{ker}(A^2)$. Then we may assume that $A$ is in a Jordan form
   $$A=\mathrm{diag}(\lambda_1I_{n_1}+N_1,\cdots,\lambda_pI_{n_p}+N_p,0_q)$$
   where $n_1\leq\cdots\leq n_p$, $q\leq{n/2}$, for every $i$, $\lambda_i\not= 0$ and $N_i$ is strictly upper triangular. If $n_1\leq q$, then we may remove, from the Jordan form of $A$, the block $A'=\mathrm{diag}(\lambda_1I_{n_1}+N_1,0_{n_1})$. Indeed $B'=\mathrm{diag}(0_{n_1},\lambda_1I_{n_1}+N_1)$ is similar to $A'$ and $A'+B'$ is invertible. Thus we may assume that $n_1>q$. If $M$ is invertible, then, since $\mathrm{char}(K)\not=2$, $M+M$ is invertible. Therefore we may assume that $A=\mathrm{diag}(\lambda_1I_{n_1}+N_1,0_q)$. The matrix $B=\mathrm{diag}(0_q,\lambda_1I_{n_1}+N_1)$ is similar to $A$ and, since $\mathrm{char}(K)\not= 2$, $A+B$ is invertible.\\
   Step 2. $A$ is nilpotent and its Jordan forms have no nilpotent Jordan blocks of dimension $1$. It is sufficient to show the result when $A=J_n$, the nilpotent Jordan block of dimension $n\geq 2$. Let $P$ be the matrix associated to the permutation $\{1\rightarrow 2,\cdots,n-1\rightarrow n,n\rightarrow 1\}$. Then $\det(AP+PA)=2^{n-2}$ is non-zero because $\mathrm{char}(K)\not=2$. Finally $B=PAP^{-1}=PAP^T$ works.\\
   Step 3. $A$ is nilpotent and $\mathrm{rank}(A)\geq n/2$. It is sufficient to show the result when $A=\mathrm{diag}(0_p,J_q)$ with $q\geq 3$ and $p\leq q-2$.\\
   $i)$ $n=p+q$ is even. Consider the matrix $P$ associated to the permutation $\{1\rightarrow n-1,\cdots,n/2-1\rightarrow n/2+1,n/2\rightarrow n,n/2+1\rightarrow 1,\cdots,n\rightarrow n/2\}$. Then $\det(AP+PA)=\pm 2^{\frac{q-p-2}{2}}$ is non-zero and $B=PAP^T$ works.\\
   $ii)$ $n$ is odd. Then $q-p\geq 3$. Consider the matrix $P$ associated to the permutation $\{1\rightarrow n-1,\cdots,\dfrac{n-1}{2}\rightarrow \dfrac{n+1}{2},\dfrac{n+1}{2}\rightarrow n,\dfrac{n+3}{2}\rightarrow 1,\cdots,n\rightarrow\dfrac{n-1}{2}\}$. Then $\det(AP+PA)=\pm 2^{\frac{q-p-3}{2}}$ is non-zero and $B=PAP^T$ works.\\   
     Step 4. $\mathrm{rank}(A)\geq n/2$. We may assume $A=\mathrm{diag}(U,V)$ where $U$ satisfies the condition of Step 1 and $V$ satisfies the conditions of Step 3. We easily conclude.\\
   $\bullet$ Case 2. (due to Florian Eisele). $K$ is not algebraically closed. Let $\overline{K}$ be an algebraic closure of $K$. We consider the rational functions 
  $$f:T=[T_{i,j}]\in GL_n(K)\rightarrow \det(A+TAT^{-1})\in K(T_{i,j}),$$
  $$\overline{f}:T=[T_{i,j}]\in GL_n(\overline{K})\rightarrow \det(A+TAT^{-1})\in \overline{K}(T_{i,j}).$$
  By Case 1, $\overline{f}$ is not the zero function. Moreover, $K$ is infinite, $GL_n(\overline{K})$ is reductive and connected. 
  That $GL_n(\overline{K})$ is connected follows by identifying it with a closed subvariety of $K^{n^2+1}$. Indeed, it can be seen as the one defined by the vanishing of the ideal generated by $\det(T)-1$ where $T$ is the ``extra'' variable in the polynomial ring. That this is a prime ideal just means that it is an irreducible polynomial, which is clear.   
  Thus, by  \cite{2}, Corollary  18.3, $GL_n(K)$ is Zariski-dense in $GL_n(\overline{K})$. So $f$ is not the zero function and we are done.
  \end{proof}
      \begin{rem}
 Assume $K$ is a subfield of $\mathbb{C}$. The linear function $\psi:X\rightarrow AX+XA$ can be written $\psi=A\otimes I_n+I_n\otimes A^T$, that is a sum of two linear functions that commute. If $\mathrm{sp}(A)=(\lambda_i)_{i\leq n}$, then $\mathrm{sp}(\psi)=(\lambda_i+\lambda_j)_{i,j\leq n}$. If at least $n^2-n+1$ among these eigenvalues are non-zero, then $\mathrm{rank}(\psi)\geq n^2-n+1$ and there exists $X$ such that $AX+XA$ is invertible (cf. \cite {3}). By a reasoning using density, we conclude that $X$ may be assumed invertible and $A+XAX^{-1}$ is invertible. This result is weaker than Theorem \ref{florian}. Indeed, if $A$ is a generic matrix satisfying $\mathrm{rank}(A)\approx n/2$, then $\mathrm{rank}(\psi)\approx n^2-(\dfrac{n}{2})^2=\dfrac{3n^2}{4}$.      
     \end{rem} 
   \begin{rem}
$\bullet$ We conjecture that Theorem \ref{florian} is true even if $K$ is a finite field. In particular, if $n\leq 3$, then we can easily show that it works for any finite field $K$.\\
 $\bullet$ In Theoerem \ref{florian}, we must assume that $\mathrm{char}(K)\not= 2$. Else, consider the matrix $A=\begin{pmatrix}I_p&0_{p,n-p}\\0_{n-p,p}&0_{n-p}\end{pmatrix}$ where $p>\dfrac{n}{2}$. If $X=\begin{pmatrix}P&Q\\R&S\end{pmatrix}$, then $AX+XA=\begin{pmatrix}0&Q\\R&0\end{pmatrix}$ cannot be invertible because $\mathrm{rank}(R)<n/2,\mathrm{rank}(Q)<n/2$.
  \end{rem}
  \begin{cor}  \label{detnul}
  Let $K$ be an infinite field such that $\mathrm{char}(K)\not= 2$. Let $A\in\mathcal{M}_n(K)$. The following conditions are equivalent\\
  $i)$ $\mathrm{rank}(A)<n/2$.\\
  $ii)$ For every $X\in\mathcal{M}_n(K)$, $\det(AX+XA)=0$.  
  \end{cor}
  \begin{proof}
  According to Lemma \ref{odd}, $i)$ implies $ii)$.\\
 Now, assume that $r=\mathrm{rank}(A)\geq\dfrac{n}{2}$. If $X\in GL_n(K)$, then $AX+XA\in GL_n(K)$ is equivalent to $A+XAX^{-1}\in GL_n(K)$. According to Theorem \ref{florian}, such a matrix $X$ exists and, consequently, $ii)$ implies $i)$.
   \end{proof}
   The next two results are valid over any field $K$.
   \begin{prop}   \label{even}
 Let $n$ be an even natural integer and $A\in\mathcal{M}_n(K)$. If $A^2=0_n$, then, for every $X\in\mathcal{M}_n(K)$, $\det(AX+XA)$ is a square in $K$. 
 \end{prop}
 \begin{proof}
 We may assume $A=\mathrm{diag}(U_1,\cdots,U_p,0_{n-2p})$ where, for every $i$, $U_i=\begin{pmatrix}0&1\\0&0\end{pmatrix}$. If $2p\leq{n-2}$, then $p=\mathrm{rank}(A)\leq\dfrac{n-2}{2}$, $\mathrm{rank}(AX+XA)\leq n-2$ and $\mathrm{adj}(AX+XA)=0_n$. Therefore we assume
  $$\mathrm{rank}(A)=\dfrac{n}{2}\text{ and }\mathrm{im} (A)=\ker (A).$$
  Put $X=[X_{i,j}]$ where, for every $i,j\leq{n/2}$, $X_{i,j}\in\mathcal{M}_2(K)$. Note that, if $x_{k,l}\in K$ is the $(k,l)$ entry of $X$, then $X_{i,j}=\begin{pmatrix}x_{2i-1,2j-1}&x_{2i-1,2j}\\x_{2i,2j-1}&x_{2i,2j}\end{pmatrix}$ and $$(AX+XA)_{i,j}=U_iX_{i,j}+X_{i,j}U_j=x_{2i,2j-1}\begin{pmatrix}1&*\\0&1\end{pmatrix}.$$
  Thus the matrices $(AX+XA)_{i,j}$ pairwise commute and $\det(AX+XA)$ depends only on the rows of even index and the columns of odd index of the matrix $X$. Let $S_q$ be the set of all permutations of $\{1,\dots,q\}$ and if $\sigma\in S_q$, then $\epsilon(\sigma)$ denotes its signature. Therefore 
  $$\det(AX+XA)=\det(\sum_{\sigma\in S_{n/2}}\epsilon(\sigma)X_{\sigma(1),1}\cdots X_{\sigma(n/2),n/2})=\det(\det(\widehat{X})\begin{pmatrix}1&*\\0&1\end{pmatrix}),$$ 
   where $\widehat{X}\in\mathcal{M}_{n/2}(K)$ is the submatrix of $X$ that is constituted by the rows of even index and the columns of odd index. Finally $\det(AX+XA)=\det^2(\widehat{X})$.
 \end{proof}
\begin{prop}   \label{twosquares}
 Let $n$ be an even integer and $\alpha\in K$ be not a square. If $A\in\mathcal{M}_n(K)$ satisfies $A^2=\alpha I_n$, then, for every $X\in\mathcal{M}_n(K)$, $\det(AX+XA)$ is in the form $r^2-\alpha s^2$, where $r,s\in K$.
 \end{prop}
 \begin{proof} Since $\alpha$ is not a square, we may assume that $A=\mathrm{diag}(U_1,\cdots,U_{n/2})$ where, for every $i\leq n/2$, $U_i=U=\begin{pmatrix}0&\alpha\\1&0\end{pmatrix}$. If  $P\in\mathcal{M}_2(K)$, then $UP+PU$ is in the form $aI_2+bU$. Put $X=[X_{i,j}]$ where, for every $i,j\leq{n/2}$, $X_{i,j}\in\mathcal{M}_2(K)$. Note that  $$(AX+XA)_{i,j}=U_iX_{i,j}+X_{i,j}U_j\text{ is in the form }a_{i,j}I_2+b_{i,j}U.$$
 Thus the matrices $(AX+XA)_{i,j}$ pairwise commute and 
 $$\det(AX+XA)=\det(\sum_{\sigma\in S_{n/2}}\epsilon(\sigma)X_{\sigma(1),1}\cdots X_{\sigma(n/2),n/2})=\det(rI_2+sU),$$
 where $r,s\in K$. Finally $\det(AX+XA)=r^2-\alpha s^2$.
 \end{proof}
 \section{When $\det(AX+XA)\geq 0$}
    In the sequel we assume that $K$ is a subfield of $\mathbb{R}$ and we study the following
 
 \medskip
 
  \textbf{Problem.} Find the matrices $A\in\mathcal{M}_n(K)$ such that, for every $X\in\mathcal{M}_n(K)$, $\det(AX+XA)\geq 0$.
  \medskip
  
   Note that  necessarily $\det(A)\geq 0$ (Choose $X=I_n$).\\
\indent The case when $n$ is odd is clear. Indeed (for every $X\in\mathcal{M}_n(K)$, $\det(AX+XA)\geq 0$) is equivalent to (for every $X\in\mathcal{M}_n(K)$, $\det(AX+XA)=0$) (to see that, change $X$ with $-X$). According to Corollary \ref{detnul}, the solutions are the matrices $A$ such that $\mathrm{rank}(A)\leq\dfrac{n-1}{2}$.\\  
 \indent  Now $n$ is assumed to be even. According to Propositions \ref{even}, \ref{twosquares} and Corollary \ref{detnul}, the matrices $A$, such that $A^2=\alpha I_n$, where $\alpha\in K,\alpha\leq 0$, or such that $\mathrm{rank}(A)\leq\dfrac{n-2}{2}$, are particular solutions. Do there exist other solutions ? The answer is no for $n=2$. More precisely, one has
 \begin{prop} \label{two} Let $A\in\mathcal{M}_2(K)$. The following conditions are equivalent\\
 $i)$ For every $X\in\mathcal{M}_2(K)$, $\det(AX+XA)\geq 0$.\\
 $ii)$ $\det(A)\geq 0$ and $\mathrm{tr}(A)=0$.\\
 $iii)$ There exists $\alpha\in K,\alpha\leq 0$, such that $A^2=\alpha I_2$.\\ 
 Moreover, the following conditions are equivalent\\
 $iv)$ For every $X\in\mathcal{M}_2(K)$, $\det(AX+XA)\leq 0$.\\
 $v)$ $A=0_2$.
 \end{prop} 
 \begin{proof}
 $\bullet$ ($i)\Rightarrow ii)$). If $A$ is a scalar matrix, then clearly $A=0_2$. Else we may assume $A=\begin{pmatrix}0&\alpha\\1&\beta\end{pmatrix}$ where $\det(A)=-\alpha\geq 0$. Let $X=\begin{pmatrix}a&1\\0&-a\end{pmatrix}$. Then, for every $a\in K$, $\det(AX+XA)=1-2a\beta\geq 0$. Thus $\beta=0$ and $\mathrm{tr}(A)=0$.\\
  $\bullet$ ($ii)\Leftrightarrow iii)$). Indeed, $ii)\Leftrightarrow (\text{ there exists }w\in\mathbb{R}\text{ such that }\mathrm{sp}(A)=\{\pm iw\})\Leftrightarrow iii)$, where $\mathrm{sp}(A)$ denotes the list of the complex eigenvalues of $A$. \\
 $\bullet$  In the same way, we show that $iv)\Leftrightarrow v)$.
 \end{proof} 
 \begin{prop}   \label{signum}
 Let $n$ be an even integer and $A\in\mathcal{M}_n(K)$ be a companion matrix such that the function $X\in\mathcal{M}_n(K)\rightarrow\det(AX+XA)$ is always non-negative or always non-positive. Then necessarily $n=2$ and, for every $X\in\mathcal{M}_2(K)$, $\det(AX+XA)\geq 0$. 
 \end{prop}
 \begin{proof}
 Let $X=[x_{i,j}]$ be a strictly upper triangular matrix and assume that $n\geq 4$. Then $\det(AX+XA)=x_{1,2}(x_{1,2}+x_{2,3})\cdots(x_{n-2,n-1}+x_{n-1,n})x_{n-1,n}$. We choose $x_{2,3}=\cdots=x_{n-1,n}=1$. Thus $\mathrm{sgn}(\det(AX+XA))=\mathrm{sgn}(x_{1,2}(x_{1,2}+1))$ can be $-1,\;0\text{ or }1$.\\
  If $n=2$, then $\det(AX+XA)={x_{1,2}}^2\geq 0$.
 \end{proof}
 We show our second main result.
  \begin{thm}  \label{main}        
  Let $n$ be an even integer and $A\in\mathcal{M}_n(K)$. The following conditions are equivalent\\
 $i)$ For every $X\in\mathcal{M}_n(K)$, $\det(AX+XA)\geq 0$.\\
  $ii)$ Either $\mathrm{rank}(A)<\dfrac{n}{2}$ or there exists $\alpha\in K,\alpha\leq 0$, such that $A^2=\alpha I_n$.\\  
 \end{thm}
 \begin{proof}
 By Propositions \ref{odd} and \ref{twosquares}, $ii)$ implies $i)$. Now on, let $A\in\mathcal{M}_n(K)$ satisfying
 $$ \text{ for every }X\in\mathcal{M}_n(K), \det(AX+XA)\geq 0 \text{ and }\mathrm{rank}(A)\geq n/2.$$
  Using the proof of Theorem \ref{florian}, Step 4, we may assume $A=\mathrm{diag}(U_r,V_s)\in\mathcal{M}_n(K)$ where $\mathrm{rank}(U)\geq r/2$, $\mathrm{ker}(U)=\mathrm{ker}(U^2)$ and $V$ is nilpotent and satisfies $\mathrm{rank}(V)\geq s/2$. According to Theorem \ref{florian}, there exist $Y_0\in GL_r(K)$ such that $\det(UY_0+Y_0U)\not= 0$ and $Z_0\in GL_s(K)$ such that $\det(VZ_0+Z_0V)\not= 0$. Considering $X$ in the form $\mathrm{diag}(Y_0,Z)$ or $\mathrm{diag}(Y,Z_0)$, we deduce that the functions $Y\rightarrow\det(UY+YU)$ and $Z\rightarrow \det(VZ+ZV)$ are both  non-negative or both non-positive and $r,s$ are even. Now on we show that $V^2=0_s$.\\
  $i)$ Assume that $V=\mathrm{diag}(0_p,J_q)$, where $q\geq 2,p\leq q-2$ and $p+q$ is even. Let $r,s\in \mathbb{Q}\subset K$ and consider the matrix $Z=[z_{i,j}]$ defined as follows
  $$\text{for every }i\in \llbracket 1,q-p-1\rrbracket,\;z_{n-i+1,2p+i}=r,$$
  $$\text{for every }i\in \llbracket 1,p+1\rrbracket,\;z_{n-q+p+3-i,i}=s,$$
  $$\text{ for every }i\in \llbracket 1,p\rrbracket,\;z_{p+1-i,p+i}=1,$$
  $$\text{the other }z_{i,j}\text{ being zero.}$$
  If $p\geq 1$ and $s=1$, then $\det(VZ+ZV)=\epsilon(p,q) r^{q-p-1}$ where $\epsilon(p,q)\in\{\pm 1\}$. Since $q-p-1$ is odd, that is contradictory. If $p=0$, $q\geq 4$ is even and $r=1$, then $\det(VZ+ZV)=1+\epsilon(q) s$ where $\epsilon(q)\in\{\pm 1\}$, that is contradictory. Finally $V=J_2$ and the function $Z\rightarrow \det(VZ+ZV)$ is non-negative.\\
  $ii)$ According to the proof of Theorem \ref{florian} Step 3, and $i)$, $V$ is similar over $K$ to $\mathrm{diag}(J_2,\cdots,J_2)$. According to Proposition \ref{even}, these matrices work and we are done. Moreover the function $Y\rightarrow\det(UY+YU)$ is non-negative.\\
  $iii)$ We may assume that $U=\mathrm{diag}(F,0_s)$ where $F\in GL_r(K)$ and $r\geq s$. Let $Y=\begin{pmatrix}F^{-1}&B\\C&D_s\end{pmatrix}$. Then $UY+YU=\begin{pmatrix}2I_r&FB\\CF&0_s\end{pmatrix}$. Moreover, if $s>0$, then $\det(UY+YU)=\epsilon(r,s) 2^{-\tau(r,s)}\det(CF^2B)$ where $\epsilon(r,s)\in\{\pm 1\}$ and $\tau(r,s)$ is a positive integer. There exists $G$, an invertible submatrix of $F^2$ of dimension $s$. To obtain $G$, we extract the lines with indices $l_1,\cdots,l_s$ of $F^2$ and the columns with indices $c_1,\cdots,c_s$ of $F^2$. Let $e_1,\cdots,e_r$ be the canonical basis of $K^r$. We choose $C=[e_{l_1},\cdots,e_{l_s}]^T,D=[e_{c_1},\cdots,e_{c_s}]$. Then $CF^2B=G$. Changing $e_{l_1}$ with $-e_{l_1}$ in $C$, we obtain $C'$. Thus $C'F^2B=G'$ is the matrix obtained from $G$ if we change the first line with its opposite. Therefore $\det(C'F^2B)=-\det(CF^2B)\not=0$, that is contradictory. We conclude that $s=0$ and $U$ is invertible.\\
  $iv)$ 
  Suppose $U\in\mathcal{M}_n(K)$ is invertible and the function $X\rightarrow\det(UX+XU)$ is non-negative. Using Frobenius decomposition over $K$, we may assume that $U=\mathrm{diag}(U_1,\cdots,U_t)$ where the $(U_i)_i$ are companion matrices of polynomials of degree $(r_i)_i$, with coefficients in $K$, with no zero roots. Choosing $X=\mathrm{diag}(\pm I_{r_1},I_{n-r_1})$, we deduce that $r_1$ is even . In the same way, all the $(r_i)_i$ are even. Choosing $X=\mathrm{diag}(Y,I_{n-r_1})$, where $Y\in\mathcal{M}_{r_1}(K)$ and using Proposition \ref{signum}, we obtain that $r_1=2$. In the same way, all the $(U_i)_i$ are $2\times 2$ matrices. According to Proposition \ref{two}, for every $i\leq t$, there exists $\lambda_i< 0$ such that ${U_i}^2=\lambda_i I_2$. It remains to show that, for instance, $\lambda_1=\lambda_2$. Choosing $X=\mathrm{diag}(Y,Z,I_{n-4})$, where $Y,Z\in\mathcal{M}_2(K)$, we may assume $n=4$ and 
 $$A=\mathrm{diag}(\begin{pmatrix}0&u\\1&0\end{pmatrix},\begin{pmatrix}0&v\\1&0\end{pmatrix}),\text{ where }u,v<0.$$
 Let $X=\begin{pmatrix}0&0&1&0\\0&0&0&1\\a&0&0&0\\0&-a+1&0&0\end{pmatrix}$. Then, for every $a\in K$,
  $$\det(AX+XA)=2(u+v)(v+a(u-v)).$$
   Clearly, $\det(AX+XA)$ has a constant signum if and only if $u=v$ and this signum is non-negative. We conclude that there exists $\alpha\in K$, $\alpha<0$ such that $U^2=\alpha I_n$. \\
   $v)$ We reduced the problem to the case $A=\mathrm{diag}(F_1,\cdots,F_r,G_1,\cdots,G_s)$ where $F_i=\begin{pmatrix}0&0\\1&0\end{pmatrix}$ and $G_j=\begin{pmatrix}0&v\\1&0\end{pmatrix}$ with $v<0$. If $r,s>0$ then $\mathrm{diag}(\begin{pmatrix}0&0\\1&0\end{pmatrix},\begin{pmatrix}0&v\\1&0\end{pmatrix})$ must be a solution. As in $iv)$, we choose $X=\begin{pmatrix}0&0&1&0\\0&0&0&1\\a&0&0&0\\0&-a+1&0&0\end{pmatrix}$ and $\det(AX+XA)=2v^2(1-a))$ has not a constant signum. Thus $r=0$ or $s=0$.   
 \end{proof}
 \begin{rem}
 In Theorem \ref{main}, we may drop the hypothesis : $n$ is even. Indeed, if $n$ is odd and $A^2=\alpha I_n$, where $\alpha\leq 0$, then necessarily $\alpha =0$ and consequently, $\mathrm{rank}(A)<n/2$.
 \end{rem}
 \begin{cor}
Let $A\in\mathcal{M}_n(K)$. Then, for every $X$, $\det(AX+XA)\leq 0$ if and only if $\mathrm{rank}(A)<n/2$. 
 \end{cor}
 \begin{proof}
 If $n$ is odd, then the result is clear. Assume that $n$ is even, $\mathrm{rank}(A)\geq n/2$ and $X\rightarrow \det(AX+XA)$ is non-positive. If we reread the proof of Theorem \ref{main}, then we conclude easily that there are no such matrices $A$. 
 \end{proof}
     

\bibliographystyle{plain}

\end{document}